\documentclass[12pt]{amsart}
\usepackage{eucal,times,amsmath,amsthm,amssymb,mathrsfs,stmaryrd,color,enumerate,accents}
\usepackage[all]{xy}
\usepackage{url}
\usepackage{fullpage}
\usepackage{float}
\usepackage{graphicx}
\usepackage{amscd}
\usepackage{amsthm}
\usepackage{amsxtra}
\usepackage{verbatim}
\usepackage{endnotes}
\usepackage{times}
\usepackage{mathtools}
\usepackage[colorlinks=true,hyperindex]{hyperref}

\theoremstyle{plain}
\newtheorem{X}{X}[section]
\newtheorem{theorem}[X]{Theorem}

\newtheorem{lemma}[X]{Lemma}

\theoremstyle{definition}
\newtheorem{definition}[X]{Definition}

\newtheorem{remark}[X]{Remark}

\begin{document}

\title{Zeta functions, Grothendieck groups, and the Witt ring}
\address{Department of Mathematics, University of Maryland, College Park, MD 20742 USA.} 
\email{atma@math.umd.edu}
\author{Niranjan Ramachandran}
\dedicatory{Dedicated to S.~Lichtenbaum on the occasion of his 75th birthday.}
\maketitle

\begin{flushright}
ÒWe shall not cease from exploration\\
And the end of all our exploring\\
Will be to arrive where we started\\
And know the place for the first time.\\
-- T.~S.~Eliot, {\it Four Quartets}
\end{flushright}

Zeta functions play a primordial role in arithmetic geometry. The aim of this paper is to provide some motivation to view zeta functions of varieties over finite fields as elements of the (big) Witt ring $W(\mathbb Z)$ of $\mathbb Z$.  Our main inspirations are 

\begin{itemize}
\item Steve Lichtenbaum's philosophy \cite{MR0406981, lichtenbaum,  MR2562461} that special values of arithmetic zeta functions and motivic L-functions are given by suitable Euler characteristics.

\item  Kazuya Kato's idea of zeta elements; Kato-Saito-Kurokawa \cite{kks} titled a chapter "$\zeta$". They say "\emph{We dropped the word "functions" because we feel more and more as we study $\zeta$ functions that $\zeta$ functions are more than just functions.}". 
 
\item The suggestion of Minhyong Kim \begin{quote} In brief, the current view is that the Iwasawa polynomial=p-adic L-function should be viewed as a path in K-theory space; see \href{http://mathoverflow.net/questions/37374/what-is-a-path-in-k-theory-space} {MO.37374}.\end{quote} 
and Steve Mitchell \cite{mitchell} \begin{quote} It is tempting to think of $KR$ as a sort of homotopical $L$-function, with $L_{K(1)} KR$ as its analytic continuation and with functional equation given by some kind of Artin-Verdier-Brown-Comenetz duality. (Although in terms of the generalized Lichtenbaum conjecture on values of $\ell$-adic $L$-functions at integer points, the values at negative integers are related to {\emph{positive}} homotopy groups of $L_{K(1)}KR$, while the values at positive integers are related to the negative homotopy groups! \end{quote} that the algebraic K-theory spectrum itself should be considered as a zeta function. 

\item M.~Kapranov's \cite{kapranov} motivic zeta function with coefficients in the Grothendieck ring of varieties and the related notion of motivic measures.

\end{itemize} 

One basic reason for an Euler-characteristic description of the special values of zeta functions is that the zeta function itself is an Euler characteristic.

There is almost nothing original in this paper. Much of this is surely known to the experts. However, except for a passing remark in  \cite{MR0364425, MR1021547, lenstra} mentioning (i) of Theorem \ref{main}, the close relations between zeta and the Witt ring do not seem to be documented  in the literature\footnote{After this article was posted to the arXiv, Antoine Chambert-Loir kindly alerted me to \cite{niko2} where Theorem \ref{main} and more is proved. See also Remark \ref{niko}.}; this provides our excuse to write this paper. Still missing is a formulation of the functional equation for the zeta function in terms of the Witt ring. We shall explore the connections with homotopy in future work.\medskip 

After preliminary definitions and a review (in the first section) of the basic structures  (such as Frobenius $F_m$ and Verschiebung $V_m$) of the Witt ring $W(R)$ of a ring $R$, we present our main results. In the second section, we show 
\begin{itemize}
\item  for varieties $X$ and $Y$ over a finite field $k= \mathbb F_q$, the zeta function $Z(X \times Y, t)$ of $X \times Y$ is the Witt product of $Z(X,t) $ and $Z(Y,t)$ in $W(\mathbb Z)$. This means that $Z(X,t)$ is a motivic measure on the Grothendieck ring of varieties over $k$. 

\item the zeta function of $X$ over $k'= \mathbb F_{q^m}$ is $F_mZ(X,t)$. 

\item If $X'$ is a variety over $k'= \mathbb F_{q^m}$ and $X$ is its Weil restriction of scalars from $k'$ to $k$, then $Z(X, t)$ contains  $V_mZ(X',t)$  in a precise sense.

\item a multiplicativity property $\zeta_P(X \times Y) = \zeta_P(X,t) *\zeta_P(Y,t)$ via Witt rings for the generating function $u_P(X,t)$ for the Poincar\'e polynomials of symmetric products of a space $X$ using a formula of Macdonald \cite{MR0143204}.  (This does not seem to have been known before, at least explicitly.)
\end{itemize}  
 We end with some interesting appearances of Witt ring in the context of  Hilbert schemes and other moduli spaces that naturally generalize the symmetric products of a quasi-projective scheme.  Remembering the result of G.~Almkvist \cite{almkvist, grayson} (see Remark \ref{almkvist}) that the Witt ring encodes the characteristic polynomial of endomorphisms, it seems now, in retrospect, that the appearance of Witt ring in zeta functions is not just unsurprising nor inevitable but rather primordial!

\section{Preliminaries}

Let $\mathbb N$ denote the set of positive integers. We write $a+_Wb$ or $^W\sum a_i$ to indicate addition in the Witt ring $W(R)$.  For any field $F$, let $Sch_F$ be the category of schemes of finite type over Spec~$F$.  A variety over $F$ is an integral scheme of finite type over ${\rm Spec}~F$. 

\subsection*{The big Witt ring $W(A)$}   \cite{cartier, bloch, MR2987372, MR1021547, larsh} \cite[Chap. IX \S1]{MR722608}. 

For any commutative ring $A$ with identity, the (big) Witt ring $W(A)$ is a commutative ring with identity defined as follows.  
The group $(W(A), +)$ is isomorphic to the group \begin{equation}\label{llambda} \Lambda(A):= (1 + tA[[t]], \times),\end{equation} a subgroup of the group of units $A[[t]]^{\times}$ (under multiplication of formal power series) of the ring $A[[t]]$.  The multiplication $*$ in $W(A)$ is uniquely determined by the requirement $$(1-at)^{-1} *(1-bt)^{-1} = (1-abt)^{-1} \qquad a,b \in A$$ and functoriality of $W(-)$: any homomorphism $f:A \to B$ induces a ring homomorphism $W(f): (A) \to W(B)$. The identity for addition $+_W$ is $1 = 1 + 0t + 0t^2 \cdots $. The identity for multiplication $*$ is $[1] =(1-t)^{-1}$; here $[1]\in W(A)$ is the image of $1 \in A$ under the   
the (multiplicative) Teichmuller map $$A \to W(A) \qquad a \mapsto [a] = (1-at)^{-1}.$$  
In particular, one has 
\begin{align}\label{prod}
(\prod_i (1-a_it)^{-1} )*(\prod_j(1-b_jt)^{-1}) &= (^W\sum_i[a_i]) *(^W\sum_j[b_j]) \nonumber\\
& = ^W\sum_{i,j} [a_ib_j]\nonumber\\
& = \prod_{i,j}(1-a_ib_jt)^{-1}.\end{align}
If $f: A \to B$ is injective, then so is $W(f): W(A) \to W(B)$.

\begin{remark} 
The construction of this ring structure on $\Lambda(A)$ comes from A.~Grothendieck's work \cite{MR0116023} on Chern classes and Riemann-Roch theory.  Given a vector bundle $V$ on a smooth proper variety $X$ over a field $F$, write $Ch(V)$ for its Chern character. Then,  $Ch(V \otimes V')$ of a tensor product is given by the Witt product of $Ch(V)$ and $Ch(V')$ in the Witt ring $W(A)$; here  $A$ is the Chow ring of $X$. 

There are four different possible definitions of the Witt ring corresponding to the four choices of the identity element $$(1\pm t)^{\pm1};$$the choice $(1+t)$ is used in the theory of Chern classes (and $\lambda$-rings - see below). The Witt ring is closely connected with the K-theory of endomorphisms; see Remark \ref{almkvist}. 
D. Kaledin \cite{kaledin} has recently provided a beautiful conceptual definition of the multiplication $*$ in $W(A)$ via Tate residues and algebraic K-theory. \qed
\end{remark} 
 Recall the identities (this will be useful in Lemma \ref{ghosty}) \begin{align*} -{\rm log}~(1-t) = \sum_{r \ge 1} \frac{t^r}{r}, &\qquad -{\rm log}~(1-bt) = \sum_{r \ge 1} b^r \frac{t^r}{r}\\ 
  t\frac{d}{dt}{\rm log}~(\frac{1}{1-bt}) & = \frac{bt}{1-bt} = bt + b^2t^2 + \cdots.\end{align*}   
The (functorial) ghost map $gh: W(A) \to A^{\mathbb N}$ is defined  as the composite 
\begin{align*} 
W(A)  \longrightarrow & \quad tA[[t]]\overset{\simeq}{\longrightarrow}  A^{\mathbb N}\\
P \mapsto t \frac{1}{P} \frac{dP}{ dt} & \qquad \sum b_rt^r  \mapsto  (b_1, b_2, \cdots).  \end{align*} 
The components of $gh(P)$  are the ghost coordinates $gh_n(P)$. Thus $$t \frac{1}{P} \frac{dP}{ dt} = \sum_{r >0} gh_r(P) t^r.$$It is clear that the ghost map is injective.    
As $$gh([b]) = ( b, b^2, b^3, \cdots), \quad gh_n([b]) = b^n,$$ the ghost map is a functorial ring homomorphism: $$gh: W(A) \to A^{\mathbb N}, \qquad gh([a][b]) = gh([a]). gh([b]).$$ 

If $\Psi: U \to U$ is an endomorphism of a finite-dimensional vector space $U$, then the ghost components of $Q(t) = \text{det}(1-t\Psi~| U)^{-1}$ are given by $gh_n(Q) = \text{Trace}(\Psi^n~|~U)$ by \cite[1.5.3]{deligne} \begin{equation}\label{weil153} t\frac{d}{dt} {\rm log}~(Q(t)) = \sum_{n\ge 0} \text{Trace}~(\Psi^n~|~U)t^n.\end{equation}

 Any $P(t) \in W(A)$ admits a unique product decomposition \begin{equation}\label{witt} P(t) = \prod_{n \ge 1}(1 - a_nt^n)^{-1} \qquad a_n \in A;\end{equation} the $a_n$'s are the Witt coordinates of $P$.

The Witt coordinates $a_j$ and the ghost coordinates $gh_n$ of any $P(t) \in W(A)$ are related by
\begin{equation}\label{ghost} gh_n = \sum_{d|n}~d.(a_d)^{n/d}.\end{equation} 
For instance, if $P = [b]$, we have $$a_1 =b, \qquad a_i =0~{\rm for}~i >1, \qquad gh_n = a_1^n = b^n.$$   

For every $n \in \mathbb N$, one has a  (Frobenius) ring homomorphism $$F_n: W(A) \to W(A) \qquad F_n([a]) = [a^n]$$ and an additive (Verschiebung) homomorphism $$V_n:W(A) \to W(A) \qquad V_n(P(t)) = P(t^n).$$ These satisfy ($P(t) \in W(A)$) 
 
\begin{itemize} 
\item  $F_n \circ F_m = F_{nm}$,  $V_n \circ V_m = V_{nm}$. 
\item $F_n \circ V_n =$ multiplication by $n$; if $m$ and $n$ are coprime, then $F_n \circ V_m = V_m \circ F_n$; if $A$ is a $\mathbb F_p$-algebra, then $V_p \circ F_p =$ multiplication by $p$.   
\item  One has $V_n([a]) = (1-at^n)^{-1}$, $V_n(P(t)) = P(t^n)$, 

\begin{equation}\label{freeb} F_m(P(t)) = ~^W \sum_{\zeta^m =1} P(\zeta t^{1/m}) =  \prod_{\zeta^m =1} P(\zeta t^{1/m}).\end{equation} 
\item The identity (\ref{witt}) becomes $P(t) = ^W \sum_{n\ge 1} V_n[a_n]$ where $a_n$ are the Witt coordinates of $P(t)$.
\item (effect on ghost coordinates) Write $g_i = gh_i(P)$. Then 
\begin{equation}\label{frob}
gh (F_n (P))  = (g_n, g_{2n}, g_{3n}, \cdots), \qquad gh(V_n (P)) = (0, \cdots, 0, ng_1, 0,\cdots, 0, ng_2, \cdots)\end{equation}
where $ng_j$ appears in $nj$'th component.  
\end{itemize}

As $$A[[t]] = \lim_{\leftarrow} \frac{A[t]}{(t^n)} = \lim_{\leftarrow} \frac{A[[t]]}{(t^n)},$$ writing $W_n(A)$ for the subgroup of units of $A[[t]]/{(t^{n+1})}$ with constant term one, we have $$W(A) = \lim_{\leftarrow} W_n(A);$$ the discrete topology on each $W_n(A)$ thus endows $W(A)$ with a topology.  The operations described above on the topological rings $W(A)$ can be described as follows \cite{grayson}.

\begin{enumerate} 
\item $gh_n: W(A) \to A$ is the unique additive continuous map which sends $[a]$ to $a^n$.
\item $F_n: W(A) \to W(A)$ is the unique additive continuous map which sends $[a]$ to $[a^n]$.
\item $V_n:W(A) \to W(A)$ is the unique additve continuous map which sends $[a]$ to $(1-at^n)^{-1}$.
\end{enumerate} 

\begin{remark}\label{almkvist} One way to think about the Witt ring is in terms of characteristic polynomials of endomorphisms. This point of view is due to G.~Almkvist \cite{almkvist} and D.~Grayson \cite{grayson}. For any commutative ring $A$ with unit, consider the category $\mathbb P_A$ of finitely generated projective $A$-modules; its Grothendieck group $ K_0(\mathbb P_A)$ is $K_0(A)$. The standard operations of linear algebra (tensor, symmetric, exterior products) endow $K_0(A)$ with the structure of a $\lambda$-ring; see below. Now consider the category $\text{End}_A$ whose objects are pairs $(P, f)$, where $P$ is a finitely generated projective $A$-module and $f: P\to P$ an endomorphism of $P$. The morphisms from $(P,f)$ to $(P', f')$ are given by $A$-module maps $g: P \to P'$ satisfying $gf = f'g$. An exact sequence in $\text{End}_A$ is one whose underlying sequence of $A$-modules is exact. Since the standard operations of linear algebra can be performed in $\text{End}_A$, the group  $K_0(\text{End}_A)$ is a $\lambda$-ring. The ideal $J$ generated by the idempotent $(A, 0)$ in $K_0(\text{End}_A)$ is isomorphic to $K_0(A)$ and we define $W'(A)$ to be the quotient $K_0(\text{End}_A)/{J}$. 

The map $$L: W'(A) \to \Lambda(A) = 1 +t A[[t]], \qquad (P,f) \mapsto {\rm det}~({\rm id}_P -tf)$$
is well-defined and an injective homomorphism of groups \cite{almkvist}.  The ring structure on $W'(A)$ comes from the tensor product of projective modules.  For any $c \in A$, write $(A, c)$ corresponding to the endomorphism $$A \to A \qquad a\mapsto ca.$$ As the tensor product of $(A,a)$ and $(A,b)$ is $(A, ab)$, $L$ becomes an injective  ring homomorphism (with dense image \cite{almkvist}) if we endow $\Lambda(A)$ with the Witt product above.  Thus, $W(A)$ is the natural receptacle for the characteristic polynomial of endomorphisms of finitely generated projective $A$-modules. 

On $W'(A)$, one has \cite{grayson} 
\begin{enumerate}
\item the ghost map $gh_n(P,f) = {\rm trace}(f^n| P)$.  See also (\ref{weil153}).
\item Frobenius $F_n(P,f) = (P, f^n)$.
\item Verschiebung $V_n(P,f) = (P^{\oplus n}, v_nf)$ where $v_nf$ is a companion matrix of order $n$ consisting of $1$'s along the sub-diagonal, $f$ in the top right corner and zeroes everywhere else. Alternatively, $V_n (P,f) = (P[x]/(x^n -f), v_nf)$ where $v_nf$ is the endomorphism $x$ on the module $P[x]/{(x^n -f)} \simeq P^{\oplus n}$; thus, $v_nf = x$ is an "$n$'th root of $f$" in some sense.
\end{enumerate} 
 
\end{remark}

 \subsection*{$\lambda$-rings} \cite{MR2891866, MR2252764,MR3065021} These were introduced by Grothendieck \cite{MR0116023} to encode the rich structure of the ring $K_0(A)$ arising from the linear algebra operations such as exterior power, symmetric powers on vector bundles. This uses the group $\Lambda(A)$ from (\ref{llambda}). 
 
 A pre-$\lambda$ ring is a  pair $(A, \lambda_t)$ of a commutative ring $A$ together with a homomorphism of groups \begin{equation}\label{lamb} \lambda_t: (A, +) \to \Lambda(A) = 1 +tA[[t]] \qquad a\mapsto \lambda_t(a) = 1 + \sum_{r \ge 1} \lambda^r(a)t^r, \qquad \lambda_1(a) =a.\end{equation} The maps $\lambda^r$ behave like "exterior power" operations; concretely, the $\lambda$-operations $\lambda^r:A \to A$ satisfy $$\lambda_t(a+b) =\lambda_t(a).\lambda_t(b), \quad \lambda_t(a) = 1 +at +\cdots,\quad \lambda^r(a+b) = \sum_{i+j =r} \lambda^i(a)\lambda^j(b).$$ Clearly, $\lambda_t(0) =1$ and $\lambda_{t}(-x) = 1/{\lambda_t(x)}$. The opposite pre-$\lambda$-ring is the pair $(A, \sigma_t)$ where 
 \begin{equation}\label{oppo}\sigma_t(a) = 1 + \sum_{r \ge 1} \sigma^r(a)t^r = \frac{1}{\lambda_{-t}(a)}.\end{equation} 
 Given a map $\lambda_t:A \to \Lambda(A)$, the Adams operations $\Psi_n: A \to A$ are defined via \begin{equation}\label{adams}t\frac{d}{dt}{\rm log}~\lambda_t(a) = t \frac{1}{\lambda_t(a)} \frac{d}{dt}\lambda_t(a) = \sum_{n\ge 1} \Psi_n(a)t^{n}. \end{equation} 
 A commutative ring $A$ is a pre-$\lambda$ ring if and only if $$\Psi_n(a) + \Psi_n(b) = \Psi_n(a+b) \quad n \ge 1.$$
 
 The ring $\mathbb Z$ is a pre-$\lambda$ ring with $\lambda_t(n) = (1+t)^n$. Also, $\mathbb R$ is a pre-$\lambda$ ring with $\lambda_t(r)$ for $r \in \mathbb R$ given by either $ (1+t)^r$ or $e^{rt}$. 
 
 A map $(A, \lambda_t) \to (A', \lambda'_t)$ of pre-$\lambda$-rings is a ring homomorphism $f: A \to A'$ such that $\Lambda_f\circ \lambda_t = \lambda'_t \circ f$ as maps from $A$ to $\Lambda(A')$; here $\Lambda_f: \Lambda(A) \to \Lambda(A')$ is the map induced by $f$. 

The group $\Lambda(A)$ becomes a ring \cite[\S2]{MR2377891} with the rule $(1+at).(1+bt) = (1+abt)$ and identity element $(1+t)$; it is a variant of our $W(A)$ - see \cite[p.58]{MR0364425}.
      
For any commutative ring $A$ with identity, there is a canonical functorial pre-$\lambda$-ring structure on $\Lambda(A)$ \cite[p. 18]{MR0364425}. 

A pre-$\lambda$ ring $(A, \lambda_t)$ is said to be a $\lambda$-ring if one of the two equivalent conditions hold
\begin{itemize} 
\item  $\lambda_t:A \to \Lambda(A)$ is a map of pre-$\lambda$ rings. 
\item Adams operations (the ghost components of $\lambda_t$) satisfy $$\Psi_n(ab) = \Psi_n(a).\Psi_n(b), \quad \Psi_n\circ \Psi_m = \Psi_{nm} \quad n,m\ge 1.$$
\end{itemize} 

Pre-$\lambda$ rings (resp. $\lambda$-rings) were previously called $\lambda$-rings (resp. special $\lambda$-rings).

 The ring $\mathbb Z$ is a $\lambda$-ring with $\lambda_t(n) = (1+t)^n$.  It is a theorem of Grothendieck that $W(A)$ is a $\lambda$-ring \cite[p.18]{MR0364425}, \cite[p.13, Proposition 1.18]{larsh}. The ring $W(\mathbb Z)$ is the free $\lambda$-ring on one generator \cite[16.74]{haz}. On $W(A)$, the maps $\lambda^r:W(A) \to W(A)$ are determined by ($[a] =(1-at)^{-1}\in W(A)$) $$\lambda^0 ([a]) = [1], \quad \lambda^1([a]) = [a], \quad \lambda^r([a]) = 1 \in W(A) (r \ge 2),.$$So $\Psi_n([a]) =F_n([a])$ (the first "Adams = Frobenius" theorem in \cite[16.22]{haz}).
 
 The forgetful functor $U$ from the category of pre-$\lambda$ rings to rings, as any forgetful functor, has a left adjoint; surprisingly, $U$ also has a right adjoint (so $U$ is compatible with limits and colimits) \cite[p.~20]{MR0364425}: $$A \mapsto \Lambda(A).$$
 This plays an important role in J.~Borger's theory \cite[p.2]{jimb}. 
    
 The ring $K_0(A)$ above is a $\lambda$-ring; here $\lambda^r(P,f)$ is given by the $r$'th exterior power $(\Lambda^rP, \Lambda^rf)$ of $(P,f)$. The opposite $\lambda$-structure on  $K_0(A)$ is given by the symmetric powers $\sigma^r (P,f) = (\text{Sym}^r P,  \text{Sym}^r f)$.  When $A$ is a field, ${\rm deg}: K_0(A) \simeq \mathbb Z$ is an isomorphism of $\lambda$-rings.  The ring $GK_F$ (see below) is a pre-$\lambda$ ring, but not, in general, a $\lambda$-ring \cite{MR1996804,MR2252764}.   

\subsection*{The Grothendieck ring of varieties}

Fix a field $F$. The Grothendieck ring $GK_F$ (often denoted\footnote{Note that there is no definition known of the higher $K_i({\rm Var}_F)$ for $i >0$.} $K_0({\rm Var}_F)$) of schemes of finite type over $F$ is defined as follows. The generators are given by the isomorphism classes $[X]$ of schemes $X$ (of finite type) over $F$ and relations are $[X-Y] +[Y] = [X]$ for every closed subscheme $Y$ of $X$ and $[X] = [X_{red}]$. The product on $GK_F$ is defined via $[X].[Y] = [X \times Y]$; the class $[{\rm Spec}~F]$ of a point is the identity for multiplication.   As quasi-projective varieties over $F$ additively generate $GK_F$, the case of quasi-projective varieties  usually suffice to prove statements about $GK_F$.  Any map of fields $F \to F'$ induces a ring homomorphism (base change) $b: GK_F \to GK_{F'}$. 

For any scheme $X$ of finite type over Spec~$\mathbb C$, we write $\chi(X)$ for the Euler characteristic for the cohomology with compact support of the topological space $X(\mathbb C)$. The map $[X] \mapsto \chi(X)$ defines a ring homomorphism $\chi: GK_{\mathbb C} \to \mathbb Z$; see below for details. Thus, for any scheme $Y$ over Spec~$F$, the element $[Y] \in GK_F$ can be viewed as \emph{the universal Euler characteristic with compact support} of $Y$.

The Kapranov zeta function (\ref{kzeta}) gives a pre-$\lambda$ ring structure on $GK_F$ via $\lambda^r([X]) = [X^{(r)}]$; here $X$ is a quasi-projective scheme and $X^{(r)}$ is the $r$'th symmetric product of $X$; in fact, there are at least four natural pre-$\lambda$ structures on $GK_F$ \cite[p. 526]{MR3065021}. 

Later, for Theorem \ref{macdonald}, we shall need a variant $GK'_F$ of $GK_F$. The ring $GK'_F$ (denoted ($\bar{K}_0(space), \cup)$ in \cite[p.299]{MR2891866}) has the same generators as $GK_F$ subject to  relations $[X] = [X_{red}]$ and  (disjoint unions): $[X\amalg Y] = [X] + [Y]$.   There is a natural quotient map $GK'_F \to GK_F$.  The group $GK'_F$ is the Grothendieck group associated with the abelian monoid of isomorphism classes of (reduced) schemes with disjoint union. The Cartesian product makes $GK'_F$ into a commutative ring. In many applications, one replaces $GK_F$ by various localizations and completions.

The genesis of $GK_F$ dates back to 1964 (it was considered by Grothendieck \cite[p.174]{grot} in his letter (dated August 16, 1964) to J.-P.~Serre; it is the first written mention of the word "motives").  The ring $GK_F$ is a shadow (decategorification) of the category of motives; some aspects of the yoga of motives are not seen at the level of $GK_F$.
We refer to  \cite[Chapter 7]{mustata} for a careful and detailed exposition of $GK_F$.

\subsection*{Schemes over finite fields and their zeta functions}  

\cite{deligne}

Let $X$ be a scheme of finite type over Spec~$\mathbb Z$, $|X|$ the set of closed points of $X$ and, for $x \in |X|$, let $N(x)$ be the cardinality of the residue field $k(x)$ of $X$ at $x$. The Hasse-Weil zeta function of $X$ is $$\zeta_X(s) = \prod_{x \in |X|} \frac{1}{(1- N(x)^{-s})} $$ which converges when the real part of $s$ is sufficiently large. Note that $\zeta_{ {\rm Spec}~Z}$ is Riemann's zeta function.

Now fix a a finite field $k = \mathbb F_q$ (here $q=p^f$) and let $X$ be a scheme of finite type over ${\rm Spec}~\mathbb F_q$. For each closed point $x$, the residue field $k(x)$ is a finite extension of $k$ (whose degree we denote by deg~$(x)$) of cardinality $q^{{\rm deg}(x)}$. The power series $$Z(X, t) = \prod_{x \in |X|} (1-t^{{\rm deg}(x)})^{-1}$$ converges for $t$ sufficiently small and one has $$Z(X, q^{-s}) = \zeta_X(s).$$ It is a theorem of B.~Dwork that $Z(X,t)$ is a rational function of $t$.  Other useful forms of $Z(X,t)$ include 
\begin{align}\label{zero} Z(X, t) & = {\rm exp} \bigg(\sum_{ r \ge 1} \#X(\mathbb F_{q^r}) \frac{t^r}{r} \bigg)\nonumber\\
& =\prod_{x \in |X|} (1-t^{{\rm deg}(x)})^{-1} \nonumber\\ & = \prod_{x \in |X|} ( 1 + t^{{\rm deg}(x)} + \cdots) \nonumber\\ & = \sum_{Y} t^{{\rm deg}(Y)}.\end{align}
Here $Y$ runs over all effective zero cycles of $X$. Recall that a zero cycle $Y =\sum_i n_i x_i$ (a finite sum) on $X$ is an element of the free abelian group generated by the closed points $x_i$ of $X$ and that $Y$ is effective if the $n_i$ are all non-negative; also, ${\rm deg}(Y) = \sum_i n_i~{\rm deg}(x_i)$. The identity (\ref{zero}) exhibits $Z(X,t)$ as a generating function of effective zero-cycles.  Thus the zeta function of $X$ depends only on the zero-cycles of $X$; in Serre's \cite{serre} terminology, $\zeta_X(s)$ depends only on the atomization of $X$. 

\subsection*{Euler characteristics} 

For any scheme $X$ over ${\rm Spec}~k$ as above, one can view $Z(X,t) \in 1 +t\mathbb Z[[t]] $ as an element of $W(\mathbb Z)$. Here are a few properties of $Z(X,t)$. 
\begin{enumerate}
\item   If $Y$ is a closed subscheme of $X$, then $Z(X,t) = Z(X-Y,t).Z(Y,t)$.  
\item $Z(X,t) = Z(X_{red}, t)$.
\item \emph{Inclusion-Exclusion Principle}: for any covering $X= Y_1 \cup \cdots \cup Y_n$ of $X$ by locally closed subschemes $Y_1, \cdots, Y_n$, one has  $$Z(X,t) = \prod_{j=1}^n (\prod_{1\le i_1 < \cdots i_j \le n} Z( Y_{i_1} \cap \cdots \cap Y_{i_j}, t)^{(-1)^{j +1}}).$$
By (2), the zeta function is insensitive to the scheme structure on the intersections. 
\end{enumerate} 

Why are the special values of $Z(X,t)$ given by Euler-characteristic formulas \cite{lichtenbaum}  (as $\ell$-adic Euler characteristics or as Weil-\'etale cohomology Euler characteristics)? Because $Z(X,t)$ itself is an Euler characteristic! 
To see this, compare the properties above of $Z(X,t)$ with the properties of the usual Euler characteristic $\chi$, say, for complex algebraic varieties (see also (\ref{zeuler})):  

\begin{itemize} 
\item If $Y$ is an closed subscheme of $X$, then $\chi(X) = \chi(X-Y) + \chi(Y)$. 
\item More generally, if $X$ is the disjoint union of $X_1$ and $X_2$, then $\chi(X) = \chi(X_1) + \chi(X_2)$. 
\item $\chi(X \times Y) = \chi(X).\chi(Y)$. 
\item For any locally trivial fiber bundle $X \to B$ with fibre $F$, one has $\chi(X) = \chi(B) \chi(F)$. 
\item $\chi(\mathbb A^n) =1$. 
\item (homotopy invariance) $\chi(X \times \mathbb A^n) = \chi(X)$ 
\end{itemize}
The zeta function satisfies analogous properties, except for homotopy invariance.  As $Z({\rm Spec}~\mathbb F_q, t) = (1-t)^{-1} = [1]$ and $Z(\mathbb A^n, t) = \frac{1}{1-q^nt} = [q^n]$,  the zeta function is clearly not homotopy invariant.  Note the identity
\begin{equation}\label{affine}Z(X \times \mathbb A^n, t) = Z(X, q^nt).\end{equation} 

The cohomological description (\ref{coh-d}) of $Z(X,t)$  is in terms of cohomology with compact support. But cohomology with compact support is  not homotopy invariant,  So we can expect the zeta function {\emph{not}} to be homotopy invariant. 

 Given a scheme $X$ over $\mathbb F_q$, we can consider its base change $X_m$ to $\mathbb F_{q^m}$ for any $m \ge 1$. The zeta function is not preserved under base change; namely, $Z(X,t)$ and $Z(X_m, t)$ are usually different. What is the relation between these functions?  What is the relation between the zeta function of a scheme $Y'$ over $k'= \mathbb F_{q^m}$ and that of its Weil restriction of scalars $Y={\rm Res}_{k'/k}Y'$, a scheme over $k=\mathbb F_q$? The properties listed above indicate that the map $ X \mapsto Z( X,t)$ is a homomorphism from $GK_{\mathbb F_q} \to W(\mathbb Z)$ of groups. Is it a ring homomorphism? 
 
 We shall see the answers in the next section.

\section{Main results}
\begin{theorem} \label{main}  

Let $X$ and $Y$ be schemes of finite type over ${\rm Spec}~k = \mathbb F_q$. 

(i) {\rm \cite[p.53]{MR0364425}, \cite[Theorem 3]{MR1021547}, \cite[p.2]{lenstra}} The zeta function of the product $X \times Y$ is the Witt product of the zeta functions of $X$ and $Y$: 
$$Z(X \times Y, t) = Z(X, t)* Z(Y, t) \in W(\mathbb Z).$$ In particular, $$Z(X^n, t) = \underbrace{Z(X,t) *\cdots * Z(X,t)}_{\text{n factors}}.$$ 

(ii) The map $$\kappa: GK_{\mathbb F_q} \to W(\mathbb Z) \qquad X \mapsto Z(X,t)$$ is a ring homomorphism. Hence $X \mapsto Z(X,t)$ is a motivic measure (see \S\ref{mm}).

(iii)  If $X \to B$ is a (Zariski locally trivial) fiber bundle with fibre $F$, namely, there is a covering of $B$ by Zariski opens $U$ with $X\times_B U$ isomorphic to $U \times_{{\rm Spec}~k} F$, then $$Z(X,t) = Z(B,t) *Z(F,t).$$ 

(iv) For any $m \in \mathbb N$, let $X_m$ be the variety over $\mathbb F_{q^m}$ obtained by base change along  $b:\mathbb F_q \to \mathbb F_{q^m}$.  One has 
$$Z(X_m/{\mathbb F_{q^m}} ,t) = F_m (Z(X/{\mathbb F_q},t)).$$

(v) One has a commutative diagram of ring homomorphisms
\[ 
\begin{CD} 
GK_{\mathbb F_q} @>{b}>> GK_{\mathbb F_{q^m}}\\
@V{\kappa}VV @VV{\kappa}V\\
W(\mathbb Z) @>{F_m}>> W(\mathbb Z).\\
\end{CD}
\]
\end{theorem}   

\begin{remark}  N.~Naumann \cite{niko2} also has proved Theorem \ref{main}; see footnote above.
 
(i)  Since $X\times_{{\rm Spec}~\mathbb F_q} {\rm Spec}~\mathbb F_q =X$, the "product" of  $Z(X,t)$ and $Z({\rm Spec}~\mathbb F_q, t)$ should be $Z(X,t)$. So $Z({\rm Spec}~\mathbb F_q, t)$ should be the identity for this "product".  As $$Z({\rm Spec}~\mathbb F_q, t) = (1-t)^{-1} = [1] \in W(\mathbb Z),$$this is highly suggestive of the Witt ring.  The identity (\ref{affine}) provides another clue:\begin{align*}
Z(X \times \mathbb A^n, t) = Z(X, q^nt) = Z(X,t)* [q^n] = Z(X,t)*Z(\mathbb A^n,t).\end{align*}

(ii) The multiplicative group $\mathbb G_m$ is the complement of a point in $\mathbb A^1$. So 
\begin{align*} Z(\mathbb G_m, t) & = Z(\mathbb A^1, t) -_W Z({\rm Spec}~\mathbb F_q, t)\\
& =  \frac{(1-t)}{(1-qt)}\\ & = [q] -_W[1] \in W(\mathbb Z).\end{align*} 
So we get $$Z(\mathbb G_m^r, t) = \underbrace{([q] -_W[1]) * \cdots *([q]-_W[1])}_{\text{r factors}}$$is the $r$'th power of $Z(\mathbb G_m,t)$ in $W(\mathbb Z)$.

(iii) (J. Parson) Consider the fibration $\mathbb A^{n+1} - {0} \to \mathbb P^n$ with fibers $\mathbb G_m$. Using Theorem \ref{main} and the Inclusion-Exclusion principle, one has \begin{align*} Z(\mathbb P^n, t) & = \frac{Z(\mathbb A^{n+1} -{0}, t)}{Z(\mathbb G_m,t)} = \frac{[q^{n+1}] -_W [1]}{[q] -_W[1]}\\ &=  [q^n] +_W \cdots +_W [1]\\& = \frac{1}{(1-q^nt) \cdots (1-t)}.\end{align*}

(iv) For certain objects $M$ in a $K$-linear rigid category $A$,  B.~Kahn \cite{kahn} has defined a motivic zeta $Z(M,t) \in 1+ tK[[t]]$; in view of  \cite[Lemma 16.2]{kahn} and our theorem, his $Z(M,t)$ is naturally an element of the Witt ring $W(K)$. 

(v) The reader will find Witt ring overtones in \cite[1.5]{deligne}, F.~Heinloth \cite[p.~1942]{MR2377891}, and in the proof of the Hasse-Davenport relations \cite[Chapter 11, \S4, p.165]{irerosen} in view of (\ref{ghost}). \qed
\end{remark} 
 
 \begin{lemma}\label{ghosty} The ghost components of
 $$P(t) = {\rm exp} \bigg(\sum_{r\ge 1} b_r  \frac{t^r}{r}\bigg) \qquad \in W(\mathbb Z)$$ are given by $$gh(P) = (b_1, b_2,b_3, \cdots). $$ \end{lemma} 
 
 \begin{proof} Direct computation:  
$$ t \frac{1}{P} \frac{dP}{ dt}  = t \frac{d{\rm log}~P}{dt} = \sum_{ r \ge 1} b_r t^r. \qquad $$\end{proof}
 
\begin{proof}  (of Theorem \ref{main})

(i) There are two ways to prove this.

 The first proof is based on the fact that the ghost map $$gh: W(\mathbb Z)  \to {\mathbb Z}^{\mathbb N}$$ is an injective ring homomorphism. Applying Lemma \ref{ghosty} to $$Z(X,t) = {\rm exp}(\sum_{ r \ge 1} \#X(\mathbb F_{q^r}) \frac{t^r}{r}),$$ we find that $ \#X(\mathbb F_{q^n})$ is the $n$'th ghost component of $Z(X,t)$.  
Now (i) follows from the identity $$\#(X \times Y) (\mathbb F_{q^n}) = \#X(\mathbb F_{q^n}). \# Y(\mathbb F_{q^n}).$$ 
  
The second proof is based on K\"{u}nneth theorem and the cohomological interpretation of the zeta function; recall \cite[1.5.4]{deligne} \begin{equation}\label{coh-d}Z(X,t) = \prod_i~{\rm det}(1-F^*t,  H^i_c(\bar{X}, \mathbb Q_{\ell}))^{(-1)^{i+1}}\end{equation} where the prime $\ell \neq$~ char~$k$ and $F^*$ is the Frobenius. Write 
\begin{equation}\label{p-eye} P_i(X,t) = {\rm det}(1-F^*t,  H^i_c(\bar{X}, \mathbb Q_{\ell})).\end{equation}  We can write $Z(X,t)$ in $W(\bar{\mathbb Q_{\ell}} )$ as a sum $\sum \pm [\alpha_X]$ over  the (inverse) eigenvalues $\alpha_X$ of Frobenius of $X$. By the K\"{u}nneth theorem, 
any $\alpha_{X \times Y}$ is a product of a $\alpha_X$ and a $\alpha_Y$.  Now use (\ref{prod}) and the fact that 
$W(\mathbb Z)$ is a subring of $W(\bar{\mathbb Q_{\ell}})$ (if the map $A \to B$ is injective, the induced map $W(A) \to W(B)$ is injective).

(ii) follows from (i). 

(iii) For an open $V \subset B$ such that the fibre bundle is trivial: $X \times_B V$ is isomorphic to $ V \times F$, one has, by (ii), $Z(V,t) = Z(B,t)*Z(F,t)$. Applying this to the open covering $U$ on which $F$ is trivial and using the inclusion-exclusion principle for the zeta function, one gets (iii). 

(iv)  Write $ g_n = \# X(\mathbb F_{q^n})$ and $h_n = \#X_m(\mathbb F_{q^{nm}})$. These are the ghost components of $Z(X,t)$ and $Z(X_m/{\mathbb F_{q^m}}, t)$ respectively. As  $$ h_n = \#X_m (\mathbb F_{q^{mn}})  = \#X(\mathbb F_{q^{mn}}) = g_{nm},$$ the definition of $F_m$ in (\ref{frob}) gives (iv).  

(v) follows from (ii) and (iv)  \end{proof}

\begin{remark} \footnote{Almost everyone I discussed this with arrived, like me, at the statement of Theorem \ref{main} via K\"unneth, but it is the first proof that is in \cite{MR0364425, MR1021547, lenstra}.}The first proof of (i) is easier and simpler than the second proof which uses standard but deep results about \'etale cohomology. There is a reason for including two proofs.  Namely, the first proof does not generalize to the noncommutative situation \cite{MR3108695} (of smooth proper DG categories over $\mathbb F_q$) where a result analogous to Theorem \ref{main} is expected to hold. The second proof may also be relevant in the context of $\Gamma$-factors; see the last section of the paper. 

Suppose $X$ is a smooth proper variety. Using (\ref{p-eye}), we can write \begin{align}\label{zeuler} Z(X,t) &= ^W\sum_i (-1)^{i} P_i(X, t)\nonumber\\ & = P_0(X,t) -_W P_1(X,t) +_WP_2(X,t)  -_W \cdots +_WP_{2{\rm dim}~X}(X,t) \end{align} as the alternating sum  in the Witt ring $W(\mathbb Z)$ of $P_i(X,t)$.  This exhibits $Z(X,t)$ as an "Euler characteristic" of $X$. This result holds for any scheme $X$ of finite type over ${\rm Spec}~\mathbb F_q$ in the larger ring $W(\mathbb Z_{\ell})$ and is expected to hold even in $W(\mathbb Z)$; it is expected but not known that $P_i(X,t) \in \mathbb Z[t]$ in general.\qed
\end{remark} 

\subsection*{Weil restriction of scalars} 

We now study the effect of Weil restriction of scalars on the zeta function.  Let $k =\mathbb F_q$ and  $G = {\rm Gal}~(\bar{k}/k)$. Write $\gamma$ for the canonical (topological) generator $x \mapsto x^q$. Fix an extension $k' =\mathbb F_{q^m} \subset \bar{k}$ and put $H = ~{\rm Gal}~(\bar{k}/{k'}) = < \gamma^m>$, a subgroup of $G$. Let $\Gamma = G/H = {\rm Gal}~(k'/k)$; the  image of $\gamma$ in $\Gamma$ is a generator (also denoted $\gamma$) of $\Gamma$.

For any scheme $X'$ of finite type over Spec~$k'$, one has a scheme $X = {\rm Res}_{k'/k}X'$ obtained by Weil restriction of scalars from $k'$ to $k$ uniquely characterized by \begin{align}\label{adj} \text{Mor}_{Sch_k}(Y, X) = \text{Mor}_{Sch_{k'}}(Y\times_{k} k', X').\end{align}  
This gives a Weil restriction functor $R_m : Sch_{k'} \to Sch_k$. If the dimension of $X'$ is $n$, then the dimension of $X = R_mX'$ is $m.n$. 

The standard description \cite{MR918564, MR1995864} of $X = R_m X'$ proceeds by showing that the product $$Y = \prod_{\sigma \in \Gamma} \sigma X'$$ of the conjugates of $X'$ can be endowed with effective descent data, i.e., the variety $Y$ over $k'$ comes from a variety $X$ over $k$. Any variety $T$ over $k$ is uniquely determined (up to isomorphism) by the pair $$(\bar{T}/\bar{k}, \pi_T)$$ of the variety $\bar{T}$ over $\bar{k}$ and the relative $q$-Frobenius $\pi_T:T \to \gamma T$ (relative to $k$).  Here the defining equations of $\gamma T$ are obtained by applying $\gamma$ to the (coefficients of the) defining equations of $T$; see \cite{MR1995864} for more details. So $X$ is pinned down by $\pi_X: X \to \gamma X$. One takes $\pi_X$ to be the map such that, on each factor $\sigma X'$,  
$$\pi_X: \sigma X' \to \gamma \sigma X'.$$ Via $X \times_{\text{Spec}~k} \text{Spec}~k' = (X')^m$, one checks that $\pi_X^m =\pi_{(X')^m}$. 

Over $\bar{k}$, the variety $\bar{X}$ is isomorphic to $(\bar{X'} )^m$. Therefore, $H^*_c(\bar{X}, \mathbb Q_{\ell})$ (as a $\mathbb Q_{\ell}$-vector space) is given by the K\"unneth theorem applied to the product variety $(\bar{X'})^m$. The Galois action on the cohomology of $\bar{X}$ is determined by the (relative) $q$-Frobenius $\pi_X$ over $k$.

\subsection*{Weil restriction and Verschiebung} 

Let us begin with  two basic examples (due to Parson) 
\begin{enumerate} 
\item if $X' = {\rm Spec}~k' = {\rm Spec}~\mathbb F_{q^m}$, then $X=R_mX'$ is $X'$ considered as a ${\rm Spec}~\mathbb F_q$-scheme. Since $Z(X,t) = (1-t^m)^{-1}$ and $Z(X',t) = (1-t)^{-1}$, we find $Z(X,t) = V_m Z(X',t)$. 

\item if $X' = \mathbb A^1$ is the affine line over ${\rm Spec}~k'$, then $X =  R_mX' \simeq \mathbb A^m$ is $m$-dimensional affine space over ${\rm Spec}~k$. So $Z(X,t) = [q^m] = (1-q^mt)^{-1}$ and $Z(X',t) = (1-q^mt)^{-1}$ are equal,  but $Z(X,t) \neq V_mZ(X', t)$. 
\end{enumerate} 

Now the Weil restriction is analogous to Verschiebung: for instance, as $X \times_k k' = (X')^m$, the composition of Weil restriction and base change transforms $X'$ to its $m$'th power is analogous to $F_m \circ V_m$ is multiplication by $m$. 
 Theorem \ref{main} (iii) relating Frobenius and base change (and atomizing) may lead one to suspect the relation $$Z(X,t) = Z(R_mX', t) = V_m Z(X',t).$$ by (see, in this regard, the discussion of $V_m$ in \cite[p.252]{grayson}) 
 \begin{align}\label{vb} 
 Z(R_m X',t) = ^W\sum_{x' \in |X'|} Z(R_m x', t) &=
  ^W\sum_{x' \in |X'|} V_m Z(x',t) = V_m~^W{}\sum_{x' \in |X'|}Z(x',t)\nonumber\\&= V_m Z(X',t).\end{align} 
Only the last equality of (\ref{vb}) is correct  as explained by the following remarks.
\begin{remark}\label{niko} (i) The Weil restriction functor $R_m$ does not give rise to a ring homomorphism $GK_{k'} \to GK_k$. Even though $R_m$ is compatible with products: $R_m(X' \times_{k'}Y') = R_mX' \times_k R_m Y'$, it is not compatible with disjoint unions. 

(ii) (Parson) The base change functor $$b: Sch_k \to Sch_{k'} \qquad X \mapsto X\times_k k'$$ has both a right and a left adjoint.  The right adjoint $R_m$ - see (\ref{adj})-  is compatible with products rather than sums (which is why the atomization argument of (\ref{vb}) is incorrect).  The left adjoint $r_m: Sch_{k'} \to Sch_k$ sends a scheme $X'$ over ${\rm Spec}~ k'$ to the scheme $X' \to {\rm Spec}~k' \to {\rm Spec}~ k$. There is a natural map from $r_m X' \to R_m X'$ which is not an isomorphism in general (check dimensions). Since $b$ has both adjoints, it is compatible with limits and colimits. 

(iii) As Verschiebung is additive, it is analogous to $r_m$;  Naumann \cite{niko2} has proved the relation $Z(r_mX, t) =  V_m Z(X,t)$. \qed \end{remark}
Although $Z(X,t)$ and $V_m Z(X't)$ are not equal in general, one has: \emph{for every integer $i$ with $0 < i \le 2\text{dim}~X'$, the polynomial $P_i(X,t)$ is divisible by $V_mP_i(X',t)$. }

\subsection*{Zeta functions and Weil restriction} 

\begin{theorem}\label{weil}  Let notations be as above. 

(a)  Let $A'$ be  an abelian variety  over $k'$. 
Let $P_1 (A',t) =  \prod_j (1-\alpha_jt)$ and $P_1(A, t)= \prod_r (1-\beta_rt)$. 
 One has $$P_1(A,t) = V_m P_1(A',t) = P_1(A', t^m) = \prod_j(1-\alpha_jt^m).$$
The set $\{\beta_1^m, \cdots\}$ coincides with the set $\{\alpha_1, \cdots\}$.
 
(b) For any smooth projective variety $X'$, one has $$P_1(X,t)  = V_m P_1(X',t).$$  
 
(c) Let $X'$ be a smooth proper geometrically connected variety over $k' = \mathbb F_{q^m}$.  For each integer $0 < i \le 2{\rm dim}~X'$, the polynomial $P_i(X,t)$ is divisible by $V_m P_i(X',t)$.  In general, $$Z(X,t) \neq V_m Z(X',t),$$ 
\begin{equation}\label{weil2} Z(X \times_k k',t) =  F_m Z(R_mX',t) = Z((X')^m, t) = \underbrace{Z(X,t) * \cdots *Z(X,t)}_{\text{m factors}}.\end{equation} 
The relation between  $$a_r =\# X'(\mathbb F_{q^r}) \quad \text{and} \quad b_r = \#X(\mathbb F_{q^{mr}})$$can be described explicitly (using $d = \text{gcd}~(m,r)$ and  $r = sd$): 
\begin{equation}\label{reln} b_r = a_{s}^d.\end{equation}
\end{theorem}

\begin{remark}  (i) The cohomology of any abelian variety is an exterior algebra on its first cohomology. So the zeta function of $A$ is determined by $P_1(A,t)$. 

(ii) Note that (b) is not true for $i=0$. If $X'$ is geometrically connected, then $X$ is geometrically connected. In this case,  $P_0(X,t) = (1-t) = P_0(X',t)$. \qed
\end{remark} 

\begin{proof}  (of Theorem \ref{weil}).

(a) For any $\ell \neq p$, the $\ell$-adic Tate module $T_{\ell}B$ of $B$ is naturally a $G$-module. One has  $$T_{\ell}B \simeq \text{Ind}~_H^G ~T_{\ell}A,$$ the induced representation of $G$ attached to the $H$-representation $T_{\ell}A$. This proves (i). Note that the identity $$V_m[1] = V_m (1-t)^{-1} = (1-t^m)^{-1} \qquad \text{(Frobenius reciprocity)} $$ actually calculates the characteristic polynomial of a generator on a representation of a cyclic group of order $m$ induced from the trivial representation of the trivial group. Because of the relation $h = g^m$ between the topological generators of $H$ and $G$, descent from $k'$ to $k$ or going from a $H$-reprsentation to a $G$-representation is like  extracting a $m$'th root. This is literally true, as every $\beta_r^m$ is an $\alpha_j$. Compare with the discussion of $V_m$ in \cite[p. 252]{grayson}, recalled in Remark \ref{almkvist}. 

(b)  This follows from the theory of the Albanese (and Picard) variety of smooth projective varieties. For any smooth projective variety $Y$ over $k$, the Tate modules of the Albanese variety $Alb(Y)$ and Picard variety $Pic^0_V$ are related to the cohomology of $Y$: (canonical isomorphisms of $G$-modules)
 $$T_{\ell}~Pic^0_Y \simeq H^1_{et}(\bar{Y}, \mathbb Z_{\ell}(1)), \qquad T_{\ell}Alb(Y) \simeq H^{2\text{dim}~Y -1}(\bar{Y}, \mathbb Z_{\ell}(\text{dim}~Y)).$$ 
 If $A'$ is the Albanese variety of $X'$, then $A$ is the Albanese variety of $X$. This follows from the functoriality of the Weil restriction. Similarly, for the Picard varieties which are the duals of $A'$ and $A$. Now (b) follows from (a) and the first canonical isomorphism above. The second canonical isomorphism, combined with (b), provides a relation between
 $P_{2\text{dim}~X' -1}(X',t)$  and $P_{2\text{dim}~X -1}(X,t)$. 

(c) Fix an integer $i$ with $0 \le i \le 2{\rm dim}~X'$. Write $h^i$ for $H^i_c(\bar{X'},\mathbb Q_{\ell})$.  Now, for $i >0$, in the Kunneth decomposition, consider the subspace $U_i \subset H^i_c(\bar{X}, \mathbb Q_{\ell})$ defined as
$$ U_i = \oplus_{j=1}^m (h^0 \otimes h^0\cdots \otimes  \underbrace{h^i}_\text{j'th component} \otimes \cdots \otimes h^0).$$
The subspace $U_i$ is a sub-$G$-representation (in fact, it is $\text{Ind}_H^G h^i$), with characteristic polynomial equal to $V_mP_i(X',t)$. This proves the required divisibility. In fact,  $P_i(X,t) = V_mP_i(X',t)$  if and only if  $U_i =  H^i_c(\bar{X}, \mathbb Q_{\ell})$.

The relation (\ref{weil2}) follows from the identity $X \times_k k' = (X')^m$ and Theorem \ref{main}. 

Finally, we turn to  the proof of (\ref{reln}). One has
\begin{align} X(\mathbb F_{q^r}) &= X'(k' \otimes_k \mathbb F_{q^r})\nonumber \\ 
&= X'(\mathbb F_{q^m} \otimes_k \mathbb F_{q^r})\nonumber \\
& = X'( \mathbb F_{q^{ms}})^d 
\end{align} 
where the first two equalities are by definition and the third by elementary Galois theory.\end{proof}

\section{Motivic measures}\label{mm} 
 
 \subsection*{Motivic measures}   \cite[\S 1]{kapranov}, \cite{MR1886763, hales, bourqui, mustata}.

Consider the category $Sch_F$ of schemes of finite type over a field $F$. For any commutative ring $R$, a motivic measure on $Sch_F$ (with values in $R$) \cite[1.1]{kapranov} is a function $\mu$ which attaches to any scheme $X$ over $F$ an element $\mu(X) \in R$. The function $\mu$ satisfies the following conditions
\begin{enumerate}
\item $\mu(X) = \mu(Y) + \mu(X-Y)$ for any closed subscheme $Y$ of $X$.
\item $\mu(X) = \mu(X_{red})$.
\item $\mu(X \times Y) = \mu(X).\mu(Y)$.
\end{enumerate}
Thus, a motivic measure on $Sch_F$ with values in $R$ is a ring homomorphism $GK_F \to R$. A weak motivic measure on $Sch_F$ with values in $R$ is a ring homomorphism $GK_F' \to R$. Any motivic measure is a weak motivic measure because of the canonical quotient map $GK'_F \to GK_F$. A weak motivic measure $\mu$ satisfies properties (1) and (3) of a measure and a weak version of (2), namely, it is additive on disjoint unions: $\mu(X \amalg Y) = \mu(X) +\mu(Y)$.  Motivic measures are invariants of algebraic varieties that behave like Euler characteristics. 

Examples: 
\begin{itemize} 
\item (the simplest measure) The dimension of an algebraic variety gives a motivic measure with values in the integral tropical ring $\mathbb T\mathbb Z$ (this is the set $\mathbb Z \cup \infty$, with addition law $+_T$ given by maximum: $a +_T b = \text{max}(a,b)$, and multiplication $*_T$ given by the usual sum: $a*_Tb = a+b$.). 
\item  The topological Euler characteristic (for cohomology with compact support) provides a measure $\chi: GK_{\mathbb C} \to \mathbb Z$. 
\item The (graded) Poincar\'e polynomial $P(X,z) = \sum_{i \ge 0} (-1)^i b_i(X) z^i$ (encoding the Betti numbers $b_i(X) = {\rm dim}_{\mathbb Q}~H^i_c(X(\mathbb C), \mathbb Q)$ of a complex algebraic scheme $X$ of finite type) gives a weak motivic measure $P: GK'_{\mathbb C} \to \mathbb Z[z]$. It is not a motivic measure on $GK_{\mathbb C}$ as it does not satisfy property (2). 
\item Theorem \ref{main} says that $X \mapsto Z(X,t)$ gives rise to a motivic measure $Z: GK_{\mathbb F_q} \to W(\mathbb Z)$.  
\end{itemize}

The classical definition of $Z(X,t)$ for schemes over finite fields was generalized by Kapranov \cite{kapranov} to schemes over a general field $F$. Fix a motivic measure $\mu:GK_F \to R$. For a quasi-projective variety $X$ over $F$, he defined the $\mu$-zeta function of $X$ as  
\begin{equation}\label{kzeta} 
\zeta_{\mu}(X,t) =  \sum_{n\ge 0} ~\mu(X^{(n)}) ~t^n \in 1 + tR[[t]],\end{equation} 
where $X^{(n)}$ is the $n$'th symmetric product of $X$. For  the measure $\chi$ on $GK_{\mathbb C}$, the associated zeta function of a point is  $$u_{\chi}(\text{point}, t) = \frac{1}{(1-t)} = [1] \in W(\mathbb Z).$$ 

Given a measure $\mu$ on $GK_F$, write $L = \mu(\mathbb A^1)$. As $(\mathbb A^1)^{(n)} = \mathbb A^n$ and $\mu(\mathbb A^n) = L^n$, one finds $$\zeta_{\mu}(\mathbb A^1, t) = \sum_{n=0}^{\infty} \mu(\mathbb A^n)~t^n = 1 + Lt + L^2t^2 + \cdots = \frac{1}{1-Lt} = [L] \in W(R).$$
The universal motivic measure on $Sch_F$ corresponding to the identity map on $GK_F$ gives rise to Kapranov's motivic zeta function of a quasi-projective scheme $X$ over $F$: \begin{equation}\label{kzetau}\zeta_u(X,t) =\sum_{n=0}^{\infty}~ [X^{(n)}]~t^n \in GK_F[[t]],\end{equation} where $[X]$ indicates the class of $X$ in $GK_F$. One can view $\zeta_u(X,t) \in W(GK_F)$. 

\begin{lemma}\label{measure} Let $F =\mathbb F_q$. 

(i)  the assignment $V \to \#V(\mathbb F_q)$ gives a measure $\mu_0$ on $GK_{\mathbb F_q}$ with values in $\mathbb Z$; 

(ii) the associated zeta function $\zeta_{\mu_0}(X, t)$ is the usual zeta function $Z(X,t)$ of $X$.
\end{lemma} 

\begin{proof} (i) clear

(ii) It suffices to show this for $X$ a quasi-projective variety over Spec~$\mathbb F_q$. We recall the proof of this well known result from \cite[p.196]{lgott}; see also \cite{mustata}. Over an algebraic closure $\bar{\mathbb F}_{q}$, the  symmetric product $\bar{X}^{(n)} = \bar{X^{(n)}}$ parametrizes effective zero cycles on $\bar{X}$.  Rational points $X^{(n)}(\mathbb F_q)$ of the $n$'th symmetric product $X^{(n)}$ correspond to effective zero cycles of degree $n$ on $X$. Now use (\ref{zero}). \end{proof} 

\begin{remark}  (i) For any quasi-projective variety $X$ over Spec~$\mathbb F_q$, one has $$Z(X,t) = \sum_{n=0}^{\infty} \#X^{(n)}(\mathbb F_q)~{t^n}.$$ 
(ii) (Parson) A simple linear-algebra analog of (i) is provided by the following.  Let $\Psi: U\to U$ be an endomorphism of a finite dimensional vector space $U$. Then \begin{equation}\label{symext} \frac{1}{\text{det}~(1-t\Psi|U)} = \sum_{n\ge 0} \text{Trace}~(\Psi~| \text{Sym}^n U) t^n.\qed
\end{equation} 
\end{remark}

\subsection*{Exponentiation of measures} 

An interesting feature of Lemma \ref{measure} is that the measure $\mu_0: GK_{\mathbb F_q} \to \mathbb Z$ gives rise to another motivic measure, namely, $Z:GK_{\mathbb F_q} \to W(\mathbb Z)$.  The Kapranov zeta function (\ref{kzetau}) is analogous to exponentiation - the product $X^n$ is analogous to $x^n$, the symmetric product $X^{(n)}$ is analogous to dividing by the term  ${n!} =$ (the size
of the symmetric group $S_n$) in the exponential function $$e^x = \sum_{n=0}^{\infty} \frac{x^n}{n!}.$$  
So the measure $Z$ corresponding to the usual zeta function $Z(X,t)$ is an "exponential" of the counting measure $\mu_0$ on $GK_{\mathbb F_q}$. 

This raises the natural question: \emph{can every (weak) motivic measure  be "exponentiated" to a (weak) motivic measure? More precisely, for any measure $\mu: GK_F \to R$, is the map $$\zeta_{\mu}: GK_F \to W(R) \qquad X \mapsto \zeta_{\mu}(X,t)$$ a ring homomorphism?  Does $\zeta_{\mu}$ give a motivic measure with values in $W(R)$? }

 The issue of exponentiation is really about compatibility of $\zeta_{\mu}$ with products as indicated by the following result (see \cite[Lemma 7.29]{mustata} for a proof; the statement has to be slightly modified if $F$ has positive characteristic. We will only need the case $F = \mathbb C$.)  
\begin{lemma}\label{add} The map $$\zeta_{\mu}: GK_F \to W(R), \qquad \qquad X \mapsto \zeta_{\mu}(X,t)$$ is a group homomorphism; for any closed subscheme $Y$ of $X$, one has $\zeta_{\mu}(X) = \zeta_{\mu}(X-Y).\zeta_{\mu}(Y)$.
\end{lemma} 

\subsection*{Macdonald's formula and exponentiation} 
 
It turns out that the motivic measure $\chi$ on $GK_{\mathbb C}$ and the weak motivic measure $P$ on $GK'_{\mathbb C}$ can be exponentiated; this follows by an application of classical formulas due to I.G.~Macdonald.

For any scheme $X$ of finite type over Spec~$\mathbb C$, the graded Poincar\'e polynomial $$P(X,z) = \sum_i (-1)^i b_i(X) z^i \in R =\mathbb Z[z]$$ encodes the Betti numbers $b_i(X) = {\rm dim}_{\mathbb Q}H^i_c(X(\mathbb C), \mathbb Q)$ for cohomology with compact support; note $\chi(X) = P(X, 1)$. 
Fix a quasi-projective variety $X$ of dimension $n$ over $\mathbb C$. Recall the classical formulas due to Macdonald \cite{MR0143204, MR1382733, MR2252764, MR2891866} which show that the (graded) Poincare polynomial $P(X,z) \in \mathbb Z[z]$ and the Euler characteristic $\chi(X)$ of $X$ determine those of the symmetric products $X^{(n)}$: 
\begin{align} \label{mac2}
\sum_{n=0}^{\infty} \chi(X^{(n)}) t^n  =  (1-t)^{-\chi(X)}= {\rm exp}\bigg(\sum_{r >0} \chi(X)~\frac{t^r}{r}\bigg), \end{align}
\begin{align}\label{mac3}
\sum_{n =0}^{\infty} P(X^{(n)})t^n & =  \frac{(1-zt)^{b_1(X)} (1-z^3t)^{b_3(X)} \cdots (1-z^{2n -1}t)^{b_{2n-1}(X)}}{(1-t)^{b_0(X)} (1-z^2t)^{b_2(X)} \cdots (1-z^{2n}t)^{b_{2n}(X)}}\nonumber\\ & = \prod_{j=0}^{j=2n} (1- z^j t)^{(-1)^{j+1} b_j(X)} = {\rm exp}\bigg(\sum_{r >0} P(X, z^r)~\frac{t^r}{r}\bigg). \end{align}
 
 \begin{theorem}\label{macdonald}  (i)  The motivic measure $$\chi: GK_{\mathbb C} \to \mathbb Z \qquad X \mapsto \chi(X)$$ exponentiates to a measure $$\zeta_{\chi}: GK_{\mathbb C} \to W(\mathbb Z). $$ 
 (ii)   The weak motivic measure $P: GK'_{\mathbb C} \to R = \mathbb Z[z]$ exponentiates to a weak motivic measure $$\zeta_P: GK'_{\mathbb C} \to W(R).$$ In particular, one has $$\zeta_P(X \times Y, t) = \zeta_P(X,t) *\zeta_P(Y,t).$$ \end{theorem}
 \begin{proof} By Lemma \ref{add}, it suffices to prove 
 $\zeta_{\chi}(X \times Y) = \zeta_{\chi}(X)*\zeta_{\chi}(Y)$ and $\zeta_P(X \times Y) = \zeta_P(X)*\zeta_P(Y)$. 
  
  (i) the identity (\ref{mac2}) reads in $W(\mathbb Z)$ as $$\zeta_{\chi}(X,t) = \sum_{n=0}^{\infty} \chi(X^{(n)}) t^n = \chi(X)[1].$$Now (i) follows from $$\zeta_{\chi}(X \times Y, t) = \chi(X \times Y)[1] = \chi(X)\chi(Y)[1] = \chi(X)[1]*\chi(Y)[1] =  \zeta_{\chi}(X)*\zeta_{\chi}(Y).$$
 
(ii) Write $R=\mathbb Z[z]$.  The motivic zeta function $$\zeta_P(X,t) = \sum_{n =0}^{\infty} P(X^{(n)})t^n \in 1 + R[[t]]$$ can be rewritten using (\ref{mac3}) as \begin{equation}\label{poin} \zeta_P(X,t)  = ^W\sum_{i=0}^{{2 \rm dim}~X} (-1)^i b_i(X) [z^{i}]  \in W(R).\end{equation}  
Since $P(X \times Y) = P(X). P(Y)$ (K\'unneth), we have $$b_k(X \times Y) = \sum_{ i=0}^{i=k} b_i(X).b_{k-i}(Y);$$ using this, we compute 
  \begin{align*}
 \zeta_P(X,t) *\zeta_P(Y,t) & =  (^W\sum_i(-1)^i b_i(X) [z^i]) *  (^W\sum_j(-1)^j b_j(Y) [z^j])\\
& = ^W\sum_{i+j} (-1)^{i+j} b_i(X)b_j(Y) [z^i]*[z^j]\\  & = ^W\sum_k (-1)^k b_k (X \times Y) [z^k]\\ & = \zeta_P(X \times Y,t).  
   \end{align*} \end{proof}

Note that the measure $\chi$ and $\zeta_{\chi}$ are obtained from $P$ and $\zeta_P$ via the map $$R =\mathbb Z[z] \to \mathbb Z \qquad u \mapsto 1.$$

\begin{remark}  (i) The K\"unneth theorem is the main ingredient in the previous proof; it also plays a crucial part in the works \cite{MR1382733, MR2252764, MR2891866}  which prove generalizations of the above Macdonald formulas for various characteristic numbers and other cohomological invariants. The multiplicativity in Theorem \ref{macdonald} also holds for these generalizations in the Witt ring over an appropriate coefficient ring.  

(ii) (Parson) We say that a Macdonald formula exists for a measure $\mu:GK_F \to R$ if  $\zeta_{\mu}(X)$ can be calculated in terms of $X$. For any measure, the existence of a Macdonald formula implies (but is not implied by)  exponentiation. We used the existence in Theorem \ref{macdonald} to prove exponentiation. Lemma \ref{measure} shows that the counting measure $\mu_0$ can be exponentiated, but there is no Macdonald formula for $\mu_0$: the zeta function $\zeta_{\mu_0}(X,t) = Z(X,t)$, in general, is not entirely determined by $\*X(\mathbb F_q)$ alone.\qed\end{remark}  

\subsection*{Zeta functions, $\lambda$-rings, power structures}  \cite{MR1996804, MR2252764, MR2377891, MR2891866, MR3065021, gorsky, bourqui} 

From the viewpoint of $\lambda$-rings, the zeta function of a variety over a finite field is defined in terms of symmetric  powers (Lemma \ref{measure}) whereas the cohomological interpretation (\ref{coh-d}, \ref{p-eye})  is in terms of exterior powers: the coefficients of the characteristic polynomial are the traces on the exterior powers.  Thus, these two have to do with opposite $\lambda$-ring structures; the nomenclature "opposite"  (\ref{oppo}) comes from the two sides of (\ref{symext}) which concern opposite $\lambda$-structures. 

Let $F$ be a field of characteristic zero. Given any measure $\mu:GK_F \to R$, the map $\zeta_{\mu}: GK_F \to \Lambda(R)$ given by the Kapranov zeta function (\ref{kzeta}) factorizes as $$GK_F \xrightarrow{\mu}  R \xrightarrow{{\hat{\zeta}_{\mu}}} \Lambda(R).$$ Lemma \ref{add} shows that $\hat{\zeta}_{\mu}$ is a homomorphism of groups and hence that the pair $(R, {\hat{\zeta}_{\mu}})$ is a pre-$\lambda$ ring. For $\mu$ the identity map on $GK_F$, we get that  $(GK_F, \zeta_u)$ is a pre-$\lambda$ ring; concretely,  the associated pre-$\lambda$ structure is defined by $\lambda^r([X]) = [X^{(r)}]$ for any quasi-projective scheme $X$. In fact, there are at least four different pre-$\lambda$ ring structures on $GK_F$ \cite[p.526]{MR3065021}. 
\footnote{The right pre-$\lambda$ structure on $GK_F$ for fields $F$ of positive characteristic is due to T.~Ekedahl \cite[p.2]{MR3065021}.}     
 
Whether $GK_F$ is a $\lambda$-ring becomes the question whether the universal measure can be exponentiated.  
It is not known whether $(GK_F, \zeta_u)$ is (not) a $\lambda$-ring in general. But the question of exponentiation could be phrased with respect to any pre-$\lambda$ structure on $GK_F$. In any case, there are motivic measures with values in $\lambda$-rings, for instance, the motivic measure with values in Chow motives \cite{MR2377891}.  Also, note that certain subrings of $GK_F$ are $\lambda$-rings; for instance, the pre-$\lambda$-subring generated by $[\mathbb A^1]$ is a $\lambda$ ring \cite[Example, p.310]{gorsky}. 

If one wishes to prove existence of a Macdonald formula (\ref{mac3})  for a general measure $GK_F \to R$, one encounters an immediate obstacle: how to make sense of symbols such as $(1-zt)^a$ for elements $a\in R$? In the above cases, $R =\mathbb Z$ and so this is not an issue. However, for general rings $R$, one needs a "power structure"  \cite{MR2252764, MR2891866, MR3065021}. 

\begin{definition}\label{power}  A power structure on a ring $R$ with identity is a map  \cite[p.526]{MR3065021} $$(1 +tR[[t]]) \times R \to 1 +tR[[t]]: (P(t), r) \mapsto (P(t))^r,$$ satisfying \begin{enumerate} 
\item $P(t)^0 =1$.
\item $(P(t))^1 = P(t)$.
\item $(P(t)Q(t))^r =(P(t))^r(Q(t))^r$.
\item $(P(t))^{ r+s} = (P(t))^r .(P(t))^s$.
\item $(P(t)^{rs} = ((P(t))^s)^r$.
 \end{enumerate}
\end{definition} 
\begin{remark} (Reinterpretation of power structures) Consider the ring ${\rm End}(\Lambda (R))$ of endomorphisms of the abelian group $\Lambda(R) = 1 +tR[[t]]$, recall that the group law is multiplication of power series; there is a natural map $$\iota: \mathbb Z \to  {\rm End}(\Lambda (R))$$ where $\iota(n)$ is the multiplication by $n$ map $P \mapsto P^n$ on $\Lambda(R)$.  We also have the Verschiebung maps $V_n \in {\rm End}(\Lambda (R))$ for $n \in \mathbb N$. 

\begin{definition}\label{powe2}  A power structure on $R$ is an extension of $\iota$  to a ring homomorphism $$ j: R \to {\rm End}(\Lambda (R)), \qquad j(r)P = P^r.$$\end{definition} 
While both definitions are equivalent, we believe Definition  \ref{powe2} to be more transparent and suggestive than Definition \ref{power}. For instance, since ${\rm End}(\Lambda (R))$  is non-commutative, one has the (conjugation) action of ${\rm Aut}(\Lambda (R))$ on the set of power structures on $R$. Namely, given a power structure $j$, the map $j_{\gamma}: R \to {\rm End}(\Lambda (R))$  defined by $j_{\gamma}(r) = (\gamma \circ j \circ \gamma^{-1})(r)$ is a ring homomorphism for each $\gamma \in {\rm Aut}(\Lambda (R))$.
Thus, $j_{\gamma}$ is a power structure on $R$. A pre-$\lambda$ ring structure on $R$ can give rise to several different power structures \cite[p. 309]{gorsky}. 

The subtlety of power structures is in the arithmetic (or torsion) of $R$ because when $R$ is a $\mathbb Q$-algebra, the logarithm and exponential functions give rise to a natural power structure \cite[p.307]{gorsky}.

Some natural power structures \cite{MR2252764} also satisfy

 (i) (normalization on the $1$-jets) $(1+t)^r = 1 +rt +$ terms of higher degree and 

 (ii) (commuting with Verschiebung maps) $(P(t^k))^r = (P(t))^r|_{t \mapsto t^k}$ .  

A power structure satisfying these additional properties is said to be finitely determined if, for any $N>0$, there exists $M >0$ such that the $N$-jet of $(P(t))^r$ is determined by the $M$-jet of $P(t)$. Such a structure is determined by the elements $(1-t)^{-r}$ for all $r\in R$ satisfying \begin{equation}\label{add2}(1-t)^{-r-s} = (1-t)^{-r}.(1-t)^{-s}.\end{equation} See \cite{MR2046199} for details. 
\qed
\end{remark}

S.~M.~Gusein-Zade, I.~Luengo and A.~Melle-Hern\'andez \cite{MR2252764} have shown how a pre-$\lambda$ ring structure $\lambda_t$ on $R$ defines a functorial \cite[Proposition 2]{MR2252764} power structure on $R$. 
The pre-$\lambda$ ring structure on $GK_F$ provided by the Kapranov zeta function $\zeta_u$ (\ref{kzetau})  is a finitely determined power structure and thus uniquely determined by the rule $$(1-t)^{-[X]} = \zeta_u(X,t) = \sum_{n=0}^{\infty}~ [X^{(n)}]~t^n \in GK_F[[t]];$$Lemma \ref{add} shows that (\ref{add2}) is satisfied.

 As pointed out in \cite[p.526]{MR3065021}, this pre-$\lambda$ structure on $GK_F$ is preferable to the others as it is defined over the Grothendieck semi-ring $GK_F^{+} \subset GK_F$ consisting of non-negative combintations of elements represented by "genuine" schemes; elements of $GK_F$ are represented by virtual sum of schemes. 

 For any complex  smooth quasi-projective variety $X$ of dimension $d$, let ${\rm Hilb}^n_X$ be the Hilbert scheme parametrizing zero-dimensional subschemes of $X$ of length $n$. Write ${\rm Hilb}^n_{X,x}$ be the subscheme of the Hilbert scheme parametrizing those subschemes supported at a given point $x \in X$.  Write $$H_X(t) = 1 + \sum_{n \ge 1} [{\rm Hilb}^n_X] t^n, \qquad  H_{X,x}(t) = 1 + \sum_{n \ge 1} [{\rm Hilb}^n_{X,x}] t^n \in \Lambda(GK_{\mathbb C}).$$
A proof of the following beautiful result can be found in \cite[Theorem 1]{MR2252764}:
$$H_X(t) = (H_{\mathbb A^d, 0}(t))^{[X]} \in \Lambda(GK_{\mathbb C}).$$
 Further applications and examples (both illustrative and interesting) of power structures can be found  in \cite{MR2252764, gorsky, bourqui}.

 \subsection*{Questions}  
   
   \emph{Does the universal measure exponentiate?} As indicated above, it seems unlikely that the universal motivic measure can be exponentiated: the ring $GK_F$ is not a $\lambda$ ring in general. Also, such an exponentiation would provide a ring homomorphism $GK_F \to W(GK_F)$ splitting the projection $g_1: W(GK_F) \to GK_F$. Such splittings could exist if $GK_F$ were a $\mathbb Q$-algebra. But $GK_{\mathbb F_q}$ is not a $\mathbb Q$-algebra as seen, for instance, by the existence of the counting measure $\mu_0$.  In the likely case that the measure does not exponentiate,  one is led to ask: Is it possible to determine $\zeta_u(X\times Y)$ from $\zeta_u(X)$ and $\zeta_u(Y)$?

\emph{Does the zeta function exponentiate?}  We saw that the measure $\zeta_{\mu_0}: GK_{\mathbb F_q} \to W(\mathbb Z)$ is the map $X \mapsto Z(X,t)$. Does the measure $Z$ exponentiate? Is there a Macdonald formula for $Z$? Namely, does $Z(X,t)$ determine $Z(X^{(n)}, t)$ for all $n >0$? 

\emph{What is the relation between the Witt ring and $\Gamma$-factors?} Consider the zeta functions $\zeta(X)$ and $\zeta(Y)$ of schemes $X$ and $Y$ of finite type over ${\rm Spec}~\mathbb Z$. Taking their product with the corresponding archimedean factors ($\Gamma$-factors) gives the completed zeta functions $\hat{\zeta}(-)$ .  Theorem \ref{main} (i) (applied at all finite primes) indicates the relation between (the non-archimedean local factors of) $\zeta(X)$, $\zeta(Y)$ and $\zeta (X \times Y)$. How about the archimedean factors? Is there an analogue for Theorem \ref{main} (i) for the $\Gamma$-factors? Can one express the $\Gamma$-factors of $X \times Y$ in terms of those of $X$ and $Y$ via a Witt-style product?  Given the description of the local factors (both archimedean and non-archimedean) in terms of regularixed determinents \cite{MR1135468} and the recent work of A.~Connes-C.~Consani \cite{cc} relating this to archimedean cyclic cohomology, it seems likely the K\"unneth theorems \cite{MR920950, MR883882} in periodic and negative cyclic cohomology  provide an analogue of Theorem \ref{main} (i) for the $\Gamma$-factors. 

\emph{What is the natural receptacle for the zeta functions of schemes over ${\rm Spec}~\mathbb Z$?}  This ring would be the global analogue of $W(\mathbb Z)$ (receptacle for the local non-archimedean factors);  the identity element would be $\hat{\zeta}({\rm Spec}~\mathbb Z)$ (the completed Riemann zeta function). Should it be a $\lambda$-ring?  In view of J.~Borger's work \cite{jimb}, it is clear that $\lambda$-rings play a prominent role in global arithmetic.
 Is there a Macdonald formula for $\zeta(X)$? for $\hat{\zeta}(X)$?   
 
 Some interesting results for $\zeta(X)$ (but not $\hat{\zeta}(X)$) have been found by J.~Elliott \cite{je}.

\subsection*{Final remarks} We end by highlighting some unnoticed appearances of the Witt ring. 

Heinloth \cite{MR2377891} has proved rationality results for the motivic zeta function with values in Chow motives for abelian varieties. This involves a particular decomposition of the zeta function $Z_X$ into $P_X$ and $Q_X$. For smooth projective varieties $X$ and $Y$, she shows that if $Z_X$ and $Z_Y$ are both rational and have functional equations, then $Z_{X \times Y}$ is rational and has a functional equation. Her proof of this beautiful result \cite[p.1942]{MR2377891} actually shows that the $P_{X\times Y}$ and $Q_{X\times Y}$ are given by Witt products involving $P_X$, $P_Y$, $Q_X$ and $Q_Y$. 

   If $X$ is a smooth algebraic variety of dimension $d$, the symmetric products are smooth for $d=1$ but not for  $d > 1$.  For surfaces, the Hilbert schemes (which are smooth) are an attractive alternate to the symmetric products.
 
For any smooth projective surface $X$ over $\mathbb F_q$, L.~G\"ottsche \cite{lgott} has shown the invariants of $X$ determine those of the Hilbert scheme $X^{[n]} = {\rm Hilb}^n(X)$.  For any variety $V$ over $\mathbb F_q$, let $e(V)$ denote the Euler characteristic of $V$, computed via $\ell$-adic cohomology. 
One of G\"ottsche's results \cite[Theorem 0.1, Identity (2)]{lgott} can be rewritten as the equality $$\sum_{n \ge 0} e(X^{[n]}) t^n =  e(X) (^W\sum_{n \ge 1}V_n[1]) \in W(\mathbb Z).$$
 
 The results of Macdonald and G\"ottsche inspired K.~Yoshioka's work \cite{y1, y2}. For any smooth projective surface $X$ over $\mathbb F_q$ and a subscheme $Y$ of $X$, Yoshioka \cite{y1} studies the number $N_{n,Y}(\mathbb F_q)$ of pairs $(Z,u)$ where $Z$ is a l.c.i. subscheme of dimension zero in $X$ of degree $n$ with support in $Y$ and $u$ is a unit in $H^0(Z, \mathcal O_Z)$.  He proves \cite[Proposition 0.2]{y1} that the associated zeta function $F_{X,Y}(t) = \sum_{n \ge 0} \#N_{n,Y}(\mathbb F_q)t^n \in 1 +t\mathbb Z[[t]]$ 
 satisfies \begin{equation}\label{yosh1}F_{X,Y}(t) = \prod_{a\ge 1} \frac{Z(Y, q^{2a-1}t^a)}{Z(Y, q^{2a-2}t^a)};\end{equation} 
this is crucial for his beautiful results on the Betti numbers of the moduli space of stable sheaves of rank two on $\mathbb P^2$. 
Using (\ref{affine}), we can rewrite Yoshioka's result above as a convergent infinite sum in $W(\mathbb Z)$:
 $$F_{X,Y}(t) = ^W\sum_{n\ge 1} V_n(Z(Y \times \mathbb A^{2n-2}, t) - Z(Y \times \mathbb A^{2n-1}, t)).$$ 
One hopes that the Witt ring can provide a conceptual explanation of these results.

\subsection*{Acknowledgements.} I would like to sincerely thank S.~Lichtenbaum for his constant support and for initiating this paper. It was he who pointed out to  me long ago that the zeta function $\zeta(X \times Y)$ of a product of varieties (over a finite field) is not the usual product $\zeta(X). \zeta(Y)$ of power series, thereby raising the question of  describing $\zeta(X \times Y)$ in terms of $\zeta (X)$ and $\zeta(Y)$. I heartily thank J.~Borger, C.~Deninger, A.~Gholampour, F.~Heinloth, L.~Hesselholt, J.~Huang, L.~Illusie, S.~Kelly, J.~Milne, J.~Rosenberg, J.~Sch\"urmann, G.~Tabuada and L.~Washington for discussions and inspiration.  This revised version of the paper owes much to a detailed and useful commentary by James Parson.  I would like to express my gratitude to him. Part of the work on this paper was conducted during a stay at the Mathematisches Institut (University of M\"unster); I thank the Institut and C.~Deninger for their kind hospitality.


\begin{thebibliography}{10}

\bibitem{almkvist}
Gert Almkvist.
\newblock The {G}rothendieck ring of the category of endomorphisms.
\newblock {\em J. Algebra}, 28:375--388, 1974.

\bibitem{bloch}
Spencer Bloch.
\newblock Algebraic {$K$}-theory and crystalline cohomology.
\newblock {\em Inst. Hautes \'Etudes Sci. Publ. Math.}, (47):187--268 (1978),
  1977.
\newblock \href{http://www.numdam.org/item?id=PMIHES_1977__47__187_0}{Link}.

\bibitem{jimb}
J.~Borger.
\newblock {L}ambda-rings and the field with one element.
\newblock 2009.
\newblock \href{http://arxiv.org/abs/0906.3146}{Link}.

\bibitem{MR722608}
Nicolas Bourbaki.
\newblock {\em \'{E}l\'ements de math\'ematique}.
\newblock Masson, Paris, 1983.
\newblock Alg{\`e}bre commutative. Chapitre 8. Dimension. Chapitre 9. Anneaux
  locaux noeth{\'e}riens complets. [Commutative algebra. Chapter 8. Dimension.
  Chapter 9. Complete Noetherian local rings].

\bibitem{bourqui}
David Bourqui.
\newblock Produit eul\'erien motivique et courbes rationnelles sur les
  vari\'et\'es toriques.
\newblock {\em Compos. Math.}, 145(6):1360--1400, 2009.

\bibitem{cartier}
Pierre Cartier.
\newblock Groupes formels associ\'es aux anneaux de {W}itt g\'en\'eralis\'es.
\newblock {\em C. R. Acad. Sci. Paris S\'er. A-B}, 265:A49--A52, 1967.
\newblock
  \href{http://gallica.bnf.fr/ark:/12148/bpt6k6236329w/f63.image.r=T265}{Link}.

\bibitem{MR1382733}
Jan Cheah.
\newblock On the cohomology of {H}ilbert schemes of points.
\newblock {\em J. Algebraic Geom.}, 5(3):479--511, 1996.

\bibitem{grot}
Pierre Colmez and Jean-Pierre Serre, editors.
\newblock {\em Correspondance {G}rothendieck-{S}erre}.
\newblock Documents Math\'ematiques (Paris) [Mathematical Documents (Paris)],
  2. Soci\'et\'e Math\'ematique de France, Paris, 2001.

\bibitem{cc}
A.~Connes and C.~Consani.
\newblock {C}yclic homology, {S}erre's local factors and the lambda-operations.
\newblock 2012.
\newblock \href{http://arxiv.org/abs/1211.4239}{Link}.

\bibitem{deligne}
Pierre Deligne.
\newblock La conjecture de {W}eil. {I}.
\newblock {\em Inst. Hautes \'Etudes Sci. Publ. Math.}, (43):273--307, 1974.
\newblock
  \href{http://www.numdam.org/numdam-bin/fitem?id=PMIHES_1974__43__273_0}{Link}.

\bibitem{MR1135468}
Christopher Deninger.
\newblock Local {$L$}-factors of motives and regularized determinants.
\newblock {\em Invent. Math.}, 107(1):135--150, 1992.

\bibitem{MR1995864}
Claus Diem and Niko Naumann.
\newblock On the structure of {W}eil restrictions of abelian varieties.
\newblock {\em J. Ramanujan Math. Soc.}, 18(2):153--174, 2003.
\href{http://front.math.ucdavis.edu/math.AG/0504359}{Link}

\bibitem{MR1021547}
Andreas W.~M. Dress and Christian Siebeneicher.
\newblock The {B}urnside ring of the infinite cyclic group and its relations to
  the necklace algebra, {$\lambda$}-rings, and the universal ring of {W}itt
  vectors.
\newblock {\em Adv. Math.}, 78(1):1--41, 1989.

\bibitem{je}
J.~Elliott.
\newblock {W}itt vectors and generalizations.
\newblock 2012.
\newblock
  \href{http://www.lemiller.net/media/slidesconf/Albuquerque.pdf}{Link}.

\bibitem{gorsky}
E.~Gorsky.
\newblock Adams operations and power structures.
\newblock {\em Mosc. Math. J.}, 9(2):305--323, back matter, 2009.

\bibitem{lgott}
Lothar G{\"o}ttsche.
\newblock The {B}etti numbers of the {H}ilbert scheme of points on a smooth
  projective surface.
\newblock {\em Math. Ann.}, 286(1-3):193--207, 1990.

\bibitem{grayson}
Daniel~R. Grayson.
\newblock Grothendieck rings and {W}itt vectors.
\newblock {\em Comm. Algebra}, 6(3):249--255, 1978.
\newblock \href{http://www.math.uiuc.edu/~dan/Papers/WittVectors.pdf}{Link}.

\bibitem{MR0116023}
Alexander Grothendieck.
\newblock La th\'eorie des classes de {C}hern.
\newblock {\em Bull. Soc. Math. France}, 86:137--154, 1958.

\bibitem{MR2046199}
S.~M. Gusein-Zade, I.~Luengo, and A.~Melle-Hern{\'a}ndez.
\newblock A power structure over the {G}rothendieck ring of varieties.
\newblock {\em Math. Res. Lett.}, 11(1):49--57, 2004.

\bibitem{MR2252764}
S.~M. Gusein-Zade, I.~Luengo, and A.~Melle-Hern{\'a}ndez.
\newblock Power structure over the {G}rothendieck ring of varieties and
  generating series of {H}ilbert schemes of points.
\newblock {\em Michigan Math. J.}, 54(2):353--359, 2006.

\bibitem{MR3065021}
S.~M. Gusein-Zade, I.~Luengo, and A.~Melle-Hern{\'a}ndez.
\newblock On the pre-{$\lambda$}-ring structure on the {G}rothendieck ring of
  stacks and the power structures over it.
\newblock {\em Bull. Lond. Math. Soc.}, 45(3):520--528, 2013.

\bibitem{hales}
Thomas~C. Hales.
\newblock What is motivic measure?
\newblock {\em Bull. Amer. Math. Soc. (N.S.)}, 42(2):119--135 (electronic),
  2005.

\bibitem{haz}
M.~Hazewinkel.
\newblock {W}itt vectors, part {I}.
\newblock 2008.
\newblock \href{http://arxiv.org/abs/0804.3888}{Link}.

\bibitem{MR2987372}
Michiel Hazewinkel.
\newblock {\em Formal groups and applications}.
\newblock AMS Chelsea Publishing, Providence, RI, 2012.
\newblock Corrected reprint of the 1978 original.

\bibitem{MR2377891}
Franziska Heinloth.
\newblock A note on functional equations for zeta functions with values in
  {C}how motives.
\newblock {\em Ann. Inst. Fourier (Grenoble)}, 57(6):1927--1945, 2007.

\bibitem{larsh}
L.~Hesselholt.
\newblock {L}ecture notes on the big de {R}ham-{W}itt complex.
\newblock 2009.
\newblock
  \href{http://www.math.nagoya-u.ac.jp/~larsh/papers/s04/drw.pdf}{Link}.

\bibitem{MR920950}
Christine~E. Hood and John D.~S. Jones.
\newblock Some algebraic properties of cyclic homology groups.
\newblock {\em $K$-Theory}, 1(4):361--384, 1987.

\bibitem{irerosen}
Kenneth Ireland and Michael Rosen.
\newblock {\em A classical introduction to modern number theory}, volume~84 of
  {\em Graduate Texts in Mathematics}.
\newblock Springer-Verlag, New York, second edition, 1990.

\bibitem{kahn}
Bruno Kahn.
\newblock Zeta functions and motives.
\newblock {\em Pure Appl. Math. Q.}, 5(1):507--570, 2009.
\newblock \href{http://www.math.jussieu.fr/~kahn/preprints/qjpam.pdf}{Link}.

\bibitem{kaledin}
D.~Kaledin.
\newblock Universal {W}itt vectors and the ``{J}apanese cocycle''.
\newblock {\em Mosc. Math. J.}, 12(3):593--604, 669, 2012.

\bibitem{kapranov}
M.~Kapranov.
\newblock The elliptic curve in the {S}-duality theory and {E}isenstein series
  for {K}ac-{M}oody groups.
\newblock 2000.
\newblock \href{http://arxiv.org/abs/math/0001005}{Link}.

\bibitem{MR883882}
Christian Kassel.
\newblock Cyclic homology, comodules, and mixed complexes.
\newblock {\em J. Algebra}, 107(1):195--216, 1987.

\bibitem{kks}
Kazuya Kato, Nobushige Kurokawa, and Takeshi Saito.
\newblock {\em Number theory. 1}, volume 186 of {\em Translations of
  Mathematical Monographs}.
\newblock American Mathematical Society, Providence, RI, 2000.
\newblock Fermat's dream, Translated from the 1996 Japanese original by Masato
  Kuwata, Iwanami Series in Modern Mathematics.

\bibitem{MR0364425}
Donald Knutson.
\newblock {\em {$\lambda $}-rings and the representation theory of the
  symmetric group}.
\newblock Lecture Notes in Mathematics, Vol. 308. Springer-Verlag, Berlin-New
  York, 1973.

\bibitem{MR1996804}
Michael Larsen and Valery~A. Lunts.
\newblock Motivic measures and stable birational geometry.
\newblock {\em Mosc. Math. J.}, 3(1):85--95, 259, 2003.

\bibitem{lenstra}
H.~Lenstra.
\newblock Construction of the ring of {W}itt vectors.
\newblock 2002.
\newblock \href{http://math.berkeley.edu/~hwl/papers/witt.pdf}{Link}.

\bibitem{lichtenbaum}
S.~Lichtenbaum.
\newblock The {W}eil-\'etale topology on schemes over finite fields.
\newblock {\em Compos. Math.}, 141(3):689--702, 2005.

\bibitem{MR0406981}
Stephen Lichtenbaum.
\newblock Values of zeta-functions, \'etale cohomology, and algebraic
  {$K$}-theory.
\newblock In {\em Algebraic {$K$}-theory, {II}: ``{C}lassical'' algebraic
  {$K$}-theory and connections with arithmetic ({P}roc. {C}onf., {B}attelle
  {M}emorial {I}nst., {S}eattle, {W}ash., 1972)}, pages 489--501. Lecture Notes
  in Math., Vol. 342. Springer, Berlin, 1973.

\bibitem{MR2562461}
Stephen Lichtenbaum.
\newblock Euler characteristics and special values of zeta-functions.
\newblock In {\em Motives and algebraic cycles}, volume~56 of {\em Fields Inst.
  Commun.}, pages 249--255. Amer. Math. Soc., Providence, RI, 2009.

\bibitem{MR1886763}
Eduard Looijenga.
\newblock Motivic measures.
\newblock {\em Ast\'erisque}, (276):267--297, 2002.
\newblock S{\'e}minaire Bourbaki, Vol. 1999/2000.

\bibitem{MR0143204}
I.~G. Macdonald.
\newblock The {P}oincar\'e polynomial of a symmetric product.
\newblock {\em Proc. Cambridge Philos. Soc.}, 58:563--568, 1962.

\bibitem{MR2891866}
Laurentiu Maxim and J{\"o}rg Sch{\"u}rmann.
\newblock Twisted genera of symmetric products.
\newblock {\em Selecta Math. (N.S.)}, 18(1):283--317, 2012.

\bibitem{mitchell}
Stephen~A. Mitchell.
\newblock {$K(1)$}-local homotopy theory, {I}wasawa theory and algebraic
  {$K$}-theory.
\newblock In {\em Handbook of {$K$}-theory. {V}ol. 1, 2}, pages 955--1010.
  Springer, Berlin, 2005.
\newblock
  \href{http://www.math.uiuc.edu/K-theory/handbook/2-0955-1010.pdf}{Link}.

\bibitem{mustata}
M.~Musta\c{t}\v{a}.
\newblock {Z}eta functions in algebraic geometry.
\newblock 2011.
\newblock \href{http://www-personal.umich.edu/~mmustata/zeta_book.pdf}{Link}.

\bibitem{niko2}
N.~Naumann.
\newblock Algebraic independence in the {G}rothendieck ring of varieties.
\newblock {\em Trans. Amer. Math. Soc.}, 359(4):1653--1683 (electronic), 2007.

\bibitem{serre}
Jean-Pierre Serre.
\newblock Zeta and {$L$} functions.
\newblock In {\em Arithmetical {A}lgebraic {G}eometry ({P}roc. {C}onf. {P}urdue
  {U}niv., 1963)}, pages 82--92. Harper \& Row, New York, 1965.

\bibitem{MR918564}
Jean-Pierre Serre.
\newblock {\em Algebraic groups and class fields}, volume 117 of {\em Graduate
  Texts in Mathematics}.
\newblock Springer-Verlag, New York, 1988.
\newblock Translated from the French.

\bibitem{MR3108695}
Gon{\c{c}}alo Tabuada.
\newblock Chow motives versus noncommutative motives.
\newblock {\em J. Noncommut. Geom.}, 7(3):767--786, 2013.

\bibitem{y2}
K{\=o}ta Yoshioka.
\newblock The {B}etti numbers of the moduli space of stable sheaves of rank
  {$2$} on {$\bold P^2$}.
\newblock {\em J. Reine Angew. Math.}, 453:193--220, 1994.

\bibitem{y1}
K{\=o}ta Yoshioka.
\newblock The {B}etti numbers of the moduli space of stable sheaves of rank
  {$2$} on a ruled surface.
\newblock {\em Math. Ann.}, 302(3):519--540, 1995.

\end{thebibliography}
\end{document}